\documentclass[11pt]{article}
\usepackage{fancyhdr}
\usepackage{graphicx}
\usepackage{amsmath,amsthm,amsfonts,amssymb,amscd}
\usepackage{commath}
\usepackage[utf8]{inputenc}
\usepackage[english]{babel}
\usepackage[headsep=0.5cm,headheight=2cm]{geometry}
\usepackage[shortlabels]{enumitem}
\usepackage{hyperref}
\usepackage{mathtools}
\usepackage{etoolbox}

\hypersetup{
	colorlinks=true,
	linkcolor=blue,
	citecolor= magenta,
	filecolor=magenta,      
	urlcolor=cyan,
	pdftitle={LUR fully closed},
	pdfpagemode=FullScreen,
}
\usepackage[backend=bibtex, style=ieee, sorting=nyt]{biblatex}
\DeclareMathOperator{\Ker}{Ker}

\DeclareMathOperator{\osc}{osc}

\DeclareMathOperator{\Span}{span}

\newtheorem{theorem}{Theorem}[section]
\newtheorem{lemma}[theorem]{Lemma}
\newtheorem{corollary}[theorem]{Corollary}

\newtheorem{proposition}[theorem]{Proposition}

\theoremstyle{definition}
\newtheorem{definition}[theorem]{Definition}
\newtheorem{notation}[theorem]{Notation}
\newtheorem{example}[theorem]{Example}
\newtheorem{question}[theorem]{Question}

\theoremstyle{remark}
\newtheorem*{remark}{Remark}

\allowdisplaybreaks

\title{Fully closed maps and LUR renormability}
\author{Todor Manev
\thanks{This study is financed by the European Union-NextGenerationEU, through the National Recovery and Resilience Plan of the Republic of Bulgaria, project № BG-RRP-2.004-0008-C01.}}
\date{\vspace{-10pt}}

\begin{document}

\maketitle

\begin{abstract}
	We show that the space of continuous functions over a compact space $X$ admits an equivalent pointwise-lowersemicontinuous locally uniformly rotund norm whenever $X$ admits a fully closed map $\pi$ onto a compact $Y$ such that $C(Y)$ and the spaces $C(\pi^{-1}(y))$ all admit such norms. A map is called fully closed if the intersection of the images of any two closed disjoint sets is finite. As a main corollary we obtain that $C(X)$ is LUR renormable whenever $X$ is a Fedorchuk compact of finite spectral height.
\end{abstract}

\setlength{\abovedisplayskip}{5pt}
\setlength{\belowdisplayskip}{5pt}

\section{Introduction} \label{sec_intro}

Recall that the norm of a Banach space $E$ is called locally uniformly rotund (LUR for short) if for any point $x$ in the unit sphere $S_E$ and any sequence $\{x_n\}_{n \in \mathbb{N}} \subset S_E$ we have that
\begin{equation*}
	\lim_{n \to \infty}\left\|\frac{x + x_n}{2}\right\| = 1 \quad \implies \quad \lim_{n \to \infty}\|x - x_n\| = 0.
\end{equation*}

There are numerous results about LUR renormings and maps that transfer such norms from one Banach space to another. For a detailed overview of the available techniques the reader is referred to \cite{dgz}, \cite{guirao2022renormings}, \cite{motv}. In this note we consider the property of LUR renormability of spaces of continuous functions on Hausdorff compact spaces.

An equivalent LUR norm can be constructed whenever $E$ admits a separable projectional resolution of the identity. For $E = C(X)$ the simplest case is that of a compact ordinal interval (see, e.g., {\cite[Theorem VII.5.1]{dgz}}). Wide classes of compacta admit a suitable transfinite sequence of retractions that corresponds to a projectional resolution of the identity. The most general of them are Valdivia compacta (see, e.g., {\cite[Definition VI.7.2, Theorem VI.7.6]{dgz}}). Positive results have been obtained for some classes for which $C(X)$ does not necessarily admit a projectional resolution of the identity. Examples include $X$ scattered with $X^{(\omega_1)} = \emptyset$ (see \cite{hay-rog_scattered}), $X$ a compact subset of functions on a Polish space, containing countably many discontinuities (see \cite{hmo_count-disc}). A particular case of the latter is the Helly compact.

A natural question to consider for $C(X)$-type spaces is the stability of the LUR renormability property under various topological transformations. One example is the fact that $C\left(\bigcup_{n =1}^{\infty} X_n\right)$ inherits an equivalent LUR norm from the spaces $C(X_n)$ (\cite{mot97}). Results in \cite{jnr1999continuous} and \cite{ribbab} give stability under products, that is $C\left(\prod_{\alpha \in \Gamma}X_\alpha\right)$ is LUR renormable, provided that each $C(X_\alpha)$ is LUR renormable.

LUR norms are often constructed through transfer techniques, that is a space $E$ can be renormed by such a norm if there exists a suitable map between $E$ and a LUR space $F$. This map does not need to be linear, as is the case for the technique developed in \cite{motv}. For a space $C(X)$ it is natural to ask if the norm can be transferred from a LUR space $C(Y)$ provided that there exists a suitable map between the compact spaces $X$ and $Y$.

The space $C(X)$ is LUR renormable if $X$ is the continuous image of a compact space for which the property holds. What can we say about the inverse implication under some additional conditions? Naturally, we may consider requiring some property of the fibers of the map. However, in \cite{kubis2011finitely} Kubi\'{s} and Molt\'{o} gave an example of a compact that admits a continuous map onto a metric space with two-point fibers that fails to have a Kadec renorming. Thus, a stability result in this direction would require stronger assumptions on the projection than simple continuity.

In this note we will consider a special class of mappings defined by Fedorchuk in \cite{fed69}. Here we will give an equivalent definition for the case that is of concern, namely mappings between Hausdorff compact spaces.

\begin{definition}\label{fc_def1}
	Let $X$ and $Y$ be Hausdorff compact spaces and $\pi$ a continuous map from $X$ onto $Y$. Then $\pi$ is called fully closed if for any closed disjoint subsets $F_1$, $F_2$ of $X$, the intersection of their images $\pi(F_1) \cap \pi(F_2)$ is finite.
\end{definition}

Notable examples of fully closed maps include the projections from the double arrow space or the lexicographic square onto the interval. More generally the projection of the lexicographic product of two totally ordered compacta onto the first coordinate is fully closed. In Section \ref{sec_fc} we present a class of compacta (Fedorchuk) that are constructed through fully closed mappings. In Section \ref{sec_concl} we show that any scattered compact can be obtained from a similar construction. 

The question of LUR renormability of $C(X)$ where the compact $X$ admits a fully closed projection onto another compact $Y$ was considered in \cite{gist}. A positive result was obtained for the case of $Y$ and the fibers of the projection being metrizable.

The natural question that arises is to find the minimal requirements for $Y$ and the fibers that would allow for a LUR renorming of $C(X)$. Evidently, $Y$ needs to be such that $C(Y)$ admits an equivalent LUR norm. As to the fibers, we cannot reach such a conclusion from the mere continuity of the map. We give a simple example in Section \ref{sec_concl} to illustrate that even full closedness does not lead to necessity of the condition. However, LUR renormability of the spaces of continuous functions over the fibers remains a natural condition to impose. We shall show that it is sufficient. The main result is thus the following: 

\begin{theorem}\label{main_thm}
	Let $X$ and $Y$ be Hausdorff compacta and $\pi$ a fully closed mapping from $X$ onto $Y$. Then $C(X)$ admits an equivalent $\tau_p$-lower semicontinuous LUR norm provided that the spaces $C(Y)$ and $C\left(\pi^{-1}(y)\right)$ for $y \in Y$ admit equivalent $\tau_p$-lsc LUR norms.
\end{theorem}

\section{Fully closed maps and Fedorchuk compacta} \label{sec_fc}

Fedorchuk introduced fully closed mappings in \cite{fed69} as they proved to be a powerful tool for constructing counterexamples in dimension theory. Let us give the original definition. We start by introducing the following notation. For $X$ and $Y$ sets, $A \subset X$ and a function $f: X \to Y$, we call the small image of $A$ under $f$, denoted $f^\#(A)$, the following subset of $Y$:
\begin{equation*}
	f^\#(A) := \{y \in Y: f^{-1}(y) \subseteq A\} = Y \setminus f\left(X \setminus A\right).
\end{equation*}

The definition of a fully closed map in its most general form is given as follows:

\begin{definition} [\cite{fedFC}] \label{fully_closed}
	Let $X$ and $Y$ be topological spaces and $f: X \to Y$ a continuous map. $f$ is called fully closed at $y \in Y$ if for any finite cover $\{U_1, ..., U_s\}$ of $f^{-1}(y)$ the set $\left(\{y\} \cup \bigcup_{i = 1}^s f^\#\left(U_i\right)\right)$ is a neighborhood of $y$. The map is simply called fully closed if it is fully closed at every point of $Y$.
\end{definition}

In \cite{ivanov1984fedorchuk} Ivanov used the construction to introduce a class of compacta known as Fedorchuk compacta. We give the definition below. For some general information about inverse systems the reader is referred to {\cite[Chapter 2.5]{eng}} or \cite{fedFC}.

\begin{definition}\label{F-comp}
	Let $S = \{X_\alpha, \pi^\beta_\alpha, \alpha, \beta \in \mu\}$ be a continuous inverse system, where $X_\alpha$ are Hausdorff compacta, the neighboring bonding mappings $\pi^{\alpha + 1}_\alpha$ are fully closed, $X_0$ is a point, and the fibers $\left(\pi^{\alpha + 1}_\alpha\right)^{-1}(y)$ are metrizable compacta for all $\alpha + 1 \in \mu$ and all $y \in X_\alpha$. Then $\lim\limits_{\leftarrow}S$ is called a Fedorchuck compact of spectral height $\mu$, provided that the system is minimal. 
\end{definition}

Below we discuss some implications of the minimality of the system, but first let us restate the result in \cite{gist} in terms of Fedorchuk compacta:

\begin{theorem}[\cite{gist}]
	$C(X)$ admits an equivalent pointwise-lower semicontinuous LUR norm whenever $X$ is a Fedorchuk compact of spectral height $3$.
\end{theorem}

As an immediate corollary of Theorem \ref{main_thm} we obtain an extension of this result. That is:

\begin{corollary}\label{finite-sh}
	$C(X)$ admits an equivalent pointwise-lower semicontinuous LUR norm whenever $X$ is a Fedorchuk compact of finite spectral height.
\end{corollary}

Let us now focus on some properties of fully closed mappings that are needed for the proof of the main theorem. First, we will discuss some useful characterizations of fully closed mappings. In order to do so, we need to introduce some notation, following \cite{fedFC}.

\begin{notation}
	Let $X$ and $Y$ be topological spaces and $\pi$ a continuous mapping from $X$ onto $Y$. If $M \subset Y$, by $Y^M$ we will denote the quotient space corresponding to the following equivalence classes:
	\begin{equation*}
		[x] \quad = \quad \left\{\begin{aligned} \quad
			&x, \qquad &x \in \pi^{-1}(M);\\
			&\pi^{-1}\left(\pi(x)\right), \qquad &\pi(x) \in Y \setminus M.
		\end{aligned}\right.
	\end{equation*}
	We will denote the corresponding quotient mapping from $X$ to $Y^M$ by $p^M$ and by $\pi^M: Y^M \to Y$ the unique mapping such that $\pi = \pi^M \circ p^M$.
	If $M = \{y\}$, $y \in Y$, we will simply write $Y^y$.
\end{notation}

In \cite{fedFC} Fedorchuk gives a list of equivalent formulations of Definition \ref{fully_closed} when the spaces are regular. In the following proposition we state the three of them that will be relevant for the proof of the main result.

\begin{proposition} [{\cite[II.1.6]{fedFC}}] \label{eqDef}
	Let $X$ and $Y$ be regular topological spaces and $f: X \to Y$ a closed map. Then the following are equivalent:
	\begin{enumerate}
		\item $f$ is fully closed. \label{def1}
		\item If $F_1, F_2 \subset X$ are closed and disjoint, the set $f(F_1) \cap f(F_2)$ is discrete. \label{def4}
		\item For any $y \in Y$, the space $Y^y$ is regular. \label{def7}
	\end{enumerate}
\end{proposition}

Some authors use (\ref{def4}) as definition of a fully closed map (as we did in Section \ref{sec_intro}) as it is the easiest to state. If the condition of regularity is replaced by compactness, we obtain precisely Definition \ref{fc_def1}. Compactness would also allow for the replacement of regular in (\ref{def7}) with Hausdorff.

The following proposition, first shown in \cite{gist}, is crucial for the proof of Theorem \ref{main_thm}, as it ties full closedness to a property of the space of continuous functions.

\begin{proposition}\label{c_0.osc}
	Let $\pi: X \to Y$ be a continous surjective map between Hausdorff compacta. Consider the linear map $T_{\pi}: C(X) \to l_\infty(Y)$ given by $T_{\pi}(f) = \left\{\osc \limits_{\pi^{-1}(y)} f\right\}_{y \in Y}$.Then $\pi$ is fully closed if and only if $T_{\pi}$ maps $C(X)$ into $c_0(Y)$.
\end{proposition}

For completeness of the exposition we will include a proof of the statement. Note that this proof is slightly different from the one given in \cite{gist}.

\begin{proof}
	Let $\pi$ be fully closed and assume that for some $f \in C(X)$ and some $\delta >0$ the set $Y_{\delta} = \left\{y \in Y: \osc \limits_{\pi^{-1}(y)} f > \delta\right\}$ is infinite. From compactness we can choose $y_0 \in Y$ to be an accumulation point of $Y_{\delta}$.
	
	For $x \in \pi^{-1}(y_0)$ choose an open neighborhood $U_x$ of $x$ such that $\osc\limits_{U_x} f < \delta$. The cover $\left\{U_x\right\}_{x \in \pi^{-1}(y_0)}$ of the compact fiber has a finite subcover, denote it $\left\{U_k\right\}_{k = 1}^n$. Then $\{y_0\} \cup \bigcup_{k = 1}^n \pi^{\#}(U_k)$ is a neighborhood of $y_0$ by the definition of full closedness.
	
	As $y_0$ is an accumulation point of $Y_{\delta}$, we can pick a point $y'$ in $Y_{\delta}$, distinct from $y_0$ in this neighborhood. Then $y' \in \pi^{\#}(U_k)$ for some $k \in \{1, \dots, n\}$. But this implies that $\pi^{-1}(y') \subset U_k$ and therefore $\osc\limits_{\pi^{-1}(y')} f \leq \osc\limits_{U_k} f < \delta$ which contradicts the choice of $y'$.
	
	To show the inverse implication assume that $\pi$ is not fully closed. By Proposition \ref{eqDef} (\ref{def4}) there exist closed disjoint sets $F_1, F_2$ in $X$ such that $f(F_1) \cap f(F_2)$ is infinite. By Urysohn's lemma we can construct a continous function $f$ on $X$ satisfying $\left.f\right\vert_{F_1} = 1$ and $\left.f\right\vert_{F_2} = 0$. Then $\osc\limits_{\pi^{-1}(y)} f = 1$ for all $y$ in the infinite set $f(F_1) \cap f(F_2)$, a contradiction.
\end{proof}

Let us discuss the minimality condition in the definition of a Fedorchuk compact. If $X = \lim\limits_{\leftarrow} \left\{X_\alpha, \pi^\alpha_\beta, \mu\right\}$ is a Fedorchuk compact, then $X$ cannot be obtained as the inverse limit of such a system of spectral height less than $\mu$. This implies that for any $\gamma$ with $\gamma + 2 \in \mu$ the projection $\pi^{\gamma + 2}_\gamma$ has at least one of the two properties:
\begin{itemize}
	\item $\left(\pi^{\gamma + 2}_\gamma\right)^{-1}(y)$ is not metrizable for some $y \in X_\gamma$.
	\item The projection $\pi^{\gamma + 2}_\gamma$ is not fully closed.
\end{itemize}

A composition of fully closed mappings is not necessarily fully closed. A simple example of this fact can be found in {\cite[II.1.12]{fedFC}}. However, the map $\pi^2_0$ sending $X_2$ to a point is always fully closed. The minimality condition then implies that $X_2$ is not metrizable. The following observation, made by Fedorchuk, gives a useful metrizability condition. 

\begin{proposition} [{\cite[Proposition II.3.10]{fedFC}}]
	Let $\pi: X \to Y$ be a fully closed mapping between Hausdorff compacta. Then $X$ is metrizable if and only if the following conditions hold:
	\begin{itemize}
		\item $Y$ is metrizable;
		\item All the fibers $\pi^{-1}(y)$ are metrizable;
		\item The set of nontrivial fibers $\left\{y \in Y: \left|\pi^{-1}(y)\right| \geq 2\right\}$ is countable.
	\end{itemize}
\end{proposition}

The first condition implies that a Fedorchuk compact $X$ of spectral height $3$ or more cannot be metrizable and consequently $C(X)$ is never separable. The third condition tells us that the number of nontrivial fibers of the map $\pi^2_1: X_2 \to X_1$ has to be uncountable.

\section{Preliminaries from renorming theory}\label{sec_norm}

Let us consider the lexicographic product $K \times L$ of two totally ordered spaces $K$ and $L$, compact in the order topology. It can be shown that the projection $\pi: K \times L \to K$ onto the first coordinate is fully closed. Notable examples are the double arrow space and the lexicographic square. In \cite{jnr95} Jayne, Namioka and Rogers proved that $C(K \times L)$ is LUR renormable, provided that $C(K)$ and $C(L)$ are LUR renormable.

The proof of this fact, in particular a version given in \cite{motv}, is the principal inspiration for the method we will use to prove Theorem \ref{main_thm} in Section \ref{sec_main}. This method relies on the fact that LUR renormability is a three-space property. More precisely, we have the following

\begin{theorem}[\cite{gtwz}, see also {\cite[Chapter 5]{motv}}] \label{3space}
	Let $X$ be a Banach space, $Y$ a closed subspace of $X$ such that $Y$ and $X/Y$ both admit equivalent LUR norms. Then $X$ is also LUR renormable.
	
	In addition, assume that the LUR norms on $Y$ and $X/Y$ are respectively $\sigma(Y,F)$ and $\sigma(X/Y, G)$-lsc for some subspaces $F \subset X^*$ and $E \subset \left(X/Y\right)^*$ . Then the LUR norm on $X$ can be chosen $\sigma(X, E)$-lower semicontinuous, where
	\begin{equation*}
		E = F + Q^* G,
	\end{equation*}
	Q being the canonical quotient map from $X$ onto $X/Y$.
\end{theorem}

Both the original proof, and the one given in \cite{motv} using a nonlinear transfer technique, rely on the Bartle-Graves selection for the quotient mapping.

Another classical result that we will need is the following: 

\begin{proposition}[see, e.g., {\cite[Chapter 5]{motv}}] \label{c_0}
	Let $\{Y_\gamma, \|\cdot\|_\gamma\}_{\gamma \in \Gamma}$ be a family of Banach spaces and let $Y$ be the space
	\begin{equation*}
		Y := \bigoplus_{c_0(\Gamma)} Y_\gamma.
	\end{equation*}
	If $\|\cdot\|_\gamma$ is LUR for all $\gamma \in \Gamma$, then $Y$ admits an equivalent LUR norm.
	
	In addition, if $\|\cdot\|_\gamma$ are $F_\gamma$-lsc for some $F_\gamma \subset Y_\gamma^*$, then the LUR norm on $Y$ can be chosen $\sigma(Y, E)$-lower semicontinuous with $E = \bigoplus_{l_1(\Gamma)} F_\gamma$.
\end{proposition}

Let us briefly discuss the lower-semicontinuity of the norm. Historically, the interest in properties of this type arizes to obtain dual renormings. An equivalent norm in a dual space is a dual norm provided that its unit ball is $w^*$-closed. In $C(X)$ spaces it is natural to consider this property for the pointwise topology. Let us mention that there is no known example of a compact space $X$ for which $C(X)$ is LUR renormable, but not with a $\tau_p$-lsc norm.

In order to work with the lower semicontinuity property, we will need the following well-known assertion:

\begin{proposition}\label{cor-lsc}
	If $\|\cdot\|$ is $\sigma(E, F)$-lower semicontinuous for some $F \subset E^*$ and $G$ is a norm-dense subspace of $F$, then $\|\cdot\|$ is $\sigma(E, G)$-lower semicontinuous.
\end{proposition}

\section{Proof of the main theorem}\label{sec_main}

We start this section with a general assertion concerning the space $Y^y$, introduced in Section \ref{sec_fc}. It states that for a continuous map between topological spaces $X$ and $Y$, $Y^y$ preserves the induced topology of the nontrivial fiber, with the only topological requirement being closedness of points in $Y$. More precisely we have the following
\begin{proposition}\label{hom-fibers}
	Let $\pi$ be a continuous mapping from a topological space $X$ onto a $T_1$ topological space $Y$ and let $y \in Y$. Then the restriction $\left.p^y \right|_{\pi^{-1}(y)} : \pi^{-1}(y) \to (\pi^y)^{-1}(y)$ is a homeomorphism.
\end{proposition}

\begin{proof}
	The map $\left.p^y \right|_{\pi^{-1}(y)} : \pi^{-1}(y) \to (\pi^y)^{-1}(y)$ is bijective by definition and is continuous as the restriction of a continuous map. We shall show that it is open.
	
	Let $V$ be a relatively open set in $\pi^{-1}(y)$. Then the set $F:= \pi^{-1}(y) \setminus V$ is relatively closed in $\pi^{-1}(y)$, which is in turn closed as the continuous preimage of a point in $Y$. $F$ is thus closed in $X$. The set $X \setminus F = \left(V \cap \pi^{-1}(y)\right) \cup \pi^{-1}\left(Y \setminus \{y\}\right)$ is therefore open in $X$.
	
	On the other hand, $X \setminus F = (p^y)^{-1}\left(p^y(V) \cup (\pi^y)^{-1}\left(Y \setminus \{y\}\right)\right)$. As $p^y$ is a quotient map, it follows that $U:= p^y(V) \cup (\pi^y)^{-1}\left(Y \setminus \{y\}\right)$ must be open in $Y^y$. As $U \cap (\pi^y)^{-1}(y) = p^y(V)$, we have that $p^y(V)$ is relatively open in $(\pi^y)^{-1}(y)$.
\end{proof}

This simple assertion allows us not to differentiate between a fiber in $X$ and the nontrivial fiber in the corresponding space $Y^y$. We use this fact in the following key lemma.

\begin{lemma}\label{keylemma}
	Let $X$ and $Y$ be Hausdorff compacta and $\pi$ a fully closed mapping from $X$ onto $Y$. Let $f \in C(\pi^{-1}(y))$. Then there exists a continuous extension $\tilde{f}$ of $f$ on the whole space $X$, such that $\osc_{\tilde{f}}\left(\pi^{-1}(y')\right) = 0$ for any $y' \neq y$.
\end{lemma}

\begin{proof}
	Let $g = f \circ (p^y)^{-1}$ be a real valued funciton on $(\pi^y)^{-1}(y)$. By Proposition \ref{hom-fibers} $g \in C\left((\pi^y)^{-1}(y)\right)$. As $(\pi^y)^{-1}(y)$ is closed in the Hausdorff compact space $Y^y$, we can construct a continuous extension $\tilde{g}$ of $g$.
	
	We can then define $\tilde{f}(x) := \tilde{g}\left(p^y (x)\right)$, which is the desired extension.
\end{proof}

\begin{remark}
	As the construction is based on the Tietze extension theorem, the function $\tilde{f}$ can be chosen such that $\|\tilde{f}\| \leq \|f\|$.
\end{remark}

\begin{corollary}\label{key_cor}
	Let $\pi: X \to Y$ be a fully closed mapping between Hausdorff compacta. Let $f_1, \dots, f_n$ be continuous functions defined on some distinct fibers $\pi^{-1}(y_1), \dots \pi^{-1}(y_n)$. Then there exists a function $\tilde{f} \in C(X)$, such that $\left.\tilde{f}\right\vert_{\pi^{-1}(y_i)} = f_i$ for $i = 1, \dots, n$ and $\osc_{\tilde{f}}\left(\pi^{-1}(y')\right) = 0$ for any $y' \notin \{y_1, \dots y_n\}$
\end{corollary}

\begin{proof}
	We can use the previous lemma to construct appropriate extensions of the functions $\tilde{f}_1, \dots, \tilde{f}_n$ locally on $\pi^{-1}(C_1), \dots, \pi^{-1}(C_n)$, where $C_1, \dots, C_n$ are closed disjoint neighborhoods of the respective points $y_1, \dots, y_n$.
	
	Then we can define the functions $g_i$ on $C_i \setminus \{y_i\}$ for $i = 1, \dots, n$ as follows: $g_i(y) = \tilde{f}_i(x)$, where $x$ is any point of the respective fiber $\pi^{-1}(y)$. By construction these functions are well defined and continuous on $\bigcup_{i = 1}^n \left(C_i \setminus \{y_i\}\right)$. As the latter is a closed subspace of the normal space $Y \setminus \left(\bigcup_{i = 1}^n y_i\right)$, there exists a continuous extension $\tilde{g}$.
	
	We define the desired extension $\tilde{f}$ as follows:
	
	\begin{equation*}
		\tilde{f}(x) \quad = \quad \left\{\begin{aligned} \quad
			&f_1(x), \qquad &x \in \pi^{-1} \left(y_1\right);\\
			&\vdots & \\
			&f_n(x), \qquad &x \in \pi^{-1} \left(y_n\right);\\
			&\tilde{g}\left(\pi(x)\right), \qquad &x \in \pi^{-1} \left(Y \setminus \bigcup_{i = 1}^n \{y_i\}\right).
		\end{aligned}\right.\qedhere
	\end{equation*}
\end{proof}

We can now proceed with the proof of Theorem \ref{main_thm}. Let us fix an arbitrary point in each fiber $\pi^{-1}(y)$ and call it $0^y$. By $H^0_y$ we will denote the subspace of $C\left(\pi^{-1}(y)\right)$ consisting of the functions satisfying $f(0_y) = 0$. We construct the $c_0$-sum of these spaces:
\begin{equation*}
	H := \bigoplus_{c_0\left(Y\right)} H^0_y.
\end{equation*}

For $f \in C(X)$, $y \in Y$ and $x \in \pi^{-1}(y)$ we define:
\begin{equation*}
	f_y(x) := \left.\left(f(x) - f(0_y)\right)\right\vert_{\pi^{-1}(y)}.
\end{equation*}

Then $f_y$ is an element of $H^0_y$ for all $y \in Y$. The main purpose of this translation is to obtain $\|f_y\| \leq \osc\limits_{\pi^{-1}(y)} f$. Then by Proposition \ref{c_0.osc} we have that $\{f_y\}_{y \in Y}$ is an element of $H$. We can therefore define the following linear map from $C(X)$ to $H$:
\begin{align*}
	T: \; C(X) &\to H\\
	f &\mapsto \{f_y\}_{y \in Y}.
\end{align*}

Proposition \ref{c_0} tells us that the space $H$ is LUR renormable. The following lemma will allow us to transfer this property to $C(X)$.

\begin{lemma} \label{T_onto}
	$T$ is onto.
\end{lemma}

\begin{proof}
	Let $h$ be an element of the unit ball of $H$. For $n \in \mathbb{N}$ we define the following sets:
	\begin{equation*}
		K_n := \left\{y \in Y : \;2^{-n} < \|h_y\| \leq 2^{1-n}\right\}.
	\end{equation*}
	Each $K_n$ is finite by definition of $H$. Thus we can enumerate its elements:
	\begin{equation*}
		K_n = \{y_{i,n}\}_{i = 1}^{i_n}.
	\end{equation*}
	We define the following functions on $\pi^{-1}(K_n)$:
	\begin{equation*}
		f_n(x) \quad = \quad \left\{\begin{aligned} \quad
			&h_{y_{1,n}}(x), \qquad &x \in \pi^{-1} \left(y_{1,n}\right);\\
			&\vdots& \\
			&h_{y_{i_n,n}}(x), \qquad &x \in \pi^{-1} \left(y_{i_n,n}\right).\\
		\end{aligned}\right.
	\end{equation*}
	
	By Corollary \ref{key_cor} we can construct continuous extensions $\tilde{f}_n$ of $f_n$ such that:
	\begin{itemize}
		\item $\osc_{\tilde{f}_n}\left(\pi^{-1}(y')\right) = 0$ for all $y' \in Y \setminus K_n$,
		\item $\left\|\tilde{f}_n\right\| \leq \left\|f_n\right\| \leq 2^{1-n}$.
	\end{itemize}
	
	We can now define $f = \sum_{n = 1}^{\infty} \tilde{f_n}$. Then $f \in C(X)$ and $Tf = h$.
\end{proof}

As a continuous surjection between Banach spaces $T$ is open by the open mapping principle. $H$ is thus isomorphic to the quotient space $X/\Ker(T)$. $\Ker(T)$ is in turn isomorphic to $C(Y)$. Thus Theorem \ref{3space} applies and we conclude the existence of an equivalent LUR norm on $C(X)$.

In addition, this norm can be chosen $\sigma(C(X), E)$-lower semicontinuous with
\begin{align*}
	&E = F + T^*G;\\
	&F = \Span \left\{\delta_x: x \in X\right\};\\
	&G = \bigoplus_{l_1(Y)} \left(\Span \left\{\delta_x: x \in \pi^{-1}(y)\right\}\right).
\end{align*}
As the pointwise topology on $C(X)$ is generated by $F$, which is norm dense in $E$, Proposition \ref{cor-lsc} gives us the desired conclusion. \qedsymbol

\section{Questions and remarks} \label{sec_concl}

Let us start this section with the example promised in the introduction. We will construct a Hausdorff compact space $K$ and a fully closed map $\pi: K \to L$ onto some other Hausdorff compact $L$, such that $C(K)$ and $C(L)$ are both LUR renormable, but the property fails for some fiber $\pi^{-1}(l)$.

\begin{example}
	Let $X$ be some Hausdorff compact for which $C(X)$ is not LUR renormable. We can find a set $\Gamma$ and embed $X$ in the cube $[0,1]^{\Gamma}$ equipped with its natural product topology. We will denote this cube $K$ and we will identify $X$ with its embedding. That is $X \subset K$.
	
	We will construct a continuous function $\pi : K \to [0,1]^{\left(K \setminus X\right)^2}$.
	
	For any pair $(x,y) \in \left(K \setminus X\right)^2$ we can find a continuous function $f_{x,y}: K \to [0,1]$, such that $\left.f_{x,y}\right\vert_{X} = 0$, $f_{x,y}(x) = 1$ and $f_{x,y}(y) = \frac{1}{2}$. Omit the last condition on the diagonal, that is $f_{x,x}(x) = 1$. We can now define a map 
	\begin{align*}
		\pi: \; K &\to [0,1]^{\left(K \setminus X\right)^2}\\
		k &\mapsto \{f_{x, y}(k)\}_{(x, y) \in \left(K \setminus X\right)^2}.
	\end{align*}
	
	This map is continuous as the coordinate functions are all continuous. Define $L$ to be the image of $K$ under $\pi$. We have now constructed our map $\pi: K \to L$.
	
	We then have $\pi^{-1}\left(\{0\}^{\left(K \setminus X\right)^2}\right) = X$ and if $x$ and $y$ are distinct points in $K \setminus X$, then $\pi(x) \neq \pi(y)$. Thus the only nontrivial fiber is $\pi^{-1}\left(\{0\}^{\left(K \setminus X\right)^2}\right)$, which makes $\pi$ fully closed.
\end{example}

Let us now go back to the question of totally ordered compacta, discussed in Section \ref{sec_norm}. We obtain as an immediate corollary of Theorem \ref{main_thm} the result that $C\left((K \times L)_{lex}\right)$ is LUR renormable, whenever $C(K)$ and $C(L)$ are. In particular, the result is true about the lexicographic cube $[0,1]^n$ for $n$ natural. In \cite{hjnr00} Haydon, Jayne, Namioka and Rogers proved that $C\left([0,1]^{\alpha}\right)$ admits an equivalent LUR norm if and only if $\alpha$ is a countable ordinal.

It is easy to see that for a limit ordinal $\alpha$ the lexicographic order topology on $[0,1]^{\alpha}$ coincides with that of the limit of the inverse system $\{[0,1]^\beta, \pi^{\gamma}_\beta, \alpha\}$, where $\pi^{\gamma}_\beta$ is the natural projection onto the first $\beta$ coordinates for $\beta \leq \gamma < \alpha$. The lexicographic cube $[0,1]^{\alpha}$ is thus a Fedorchuk compact of spectral height $\alpha$ (see \cite{watson1992construction}). The natural question that arises is the following:

\begin{question} \label{fed_countable}
	Let $X$ be a Fedorchuk compact of countable spectral height. Does $C(X)$ admit an equivalent pointwise-lowersemicontinuous LUR norm?
\end{question} 

We will now consider an application of our main theorem. What follows is a well-known result that can be easily obtained using Theorem \ref{3space} (see e.g. {\cite[Corollary VII.4.6]{dgz}}). Here we get it directly from Theorem \ref{main_thm}.

\begin{corollary} \label{derived}
	Let $K$ be a Hausdorff compact such that $C(K')$ admits an equivalent $\tau_p$-lsc LUR norm. Then $C(K)$ also admits such a norm.
\end{corollary}

\begin{proof}
	Define $\pi$ from $K$ onto the one-point compactification $\alpha \left(K \setminus K'\right)$ in the following way:
	\begin{equation*}
		\pi(x) \quad = \quad \left\{\begin{aligned} \quad
			&x, \qquad &x \in K \setminus K'\\
			&\infty \qquad &x \in K'.
		\end{aligned}\right.
	\end{equation*}
	This projection is continuous and it is fully closed as the only nontrivial fiber is $\pi^{-1}(\infty) = K'$. The space $C\left(\alpha \left(K \setminus K'\right)\right)$ is isomorphic to $c_0\left(K \setminus K'\right)$ and Theorem \ref{main_thm} applies.
\end{proof}

\begin{remark}
	As a direct consequence we obtain the result of Deville from \cite{dev_scattered} that $C(K)$ is LUR renormable whenever $K$ is a scattered compact with $K^{(\omega)} = \emptyset$. In particular, we obtain LUR renormability of $C(K)$ for $K$ the Ciesielski-Pol compact (see {\cite[Examples VI.8.8]{dgz}}). The latter is of particular significance in renorming theory as $C(K)$ is LUR renormable, but there is no linear bounded one-to-one operator from $C(K)$ into $c_0(\Gamma)$ for any set $\Gamma$.
\end{remark}

Now let $K$ be a scattered compact with $K^{(\gamma)} = \emptyset$ for some ordinal $\gamma$. For $\beta < \gamma$ we consider the spaces $X_\beta := \alpha\left(K \setminus K^{(\beta)}\right)$ and the following projections:
\begin{align*}
	\pi^{\beta + 1}_{\beta}: X_{\beta + 1} &\to X_{\beta}\\
	\pi^{\beta + 1}_{\beta}(x) \quad = \quad &\left\{\begin{aligned} \quad
		&x, \qquad &x \in X_{\beta}\\
		&\infty \qquad &x \in X_{\beta + 1} \setminus X_{\beta}.
	\end{aligned}\right.
\end{align*}
As above, these maps are continuous and fully closed for all $\beta < \gamma$. Moreover, for each of the projections the only nontrivial fiber is homeomorphic to the one-point compactification $\alpha\left(K^{(\beta)} \setminus K^{(\beta + 1)}\right)$ of a discrete set. We can now define the inverse system $S = \left\{X_\beta, \pi^{\alpha}_ \beta, \alpha, \beta \in \gamma\right\}$. It can be shown that $S$ is a continuous inverse system and its limit $\lim\limits_{\leftarrow} S \cong K$.

In \cite{hay-rog_scattered} Haydon and Rogers showed that $C(K)$ is LUR renormable whenever $K^{(\omega_1)} = \emptyset$ (see also {\cite[Corollary VII.4.5]{dgz}}). In this case $K$ can be reconstructed as above as the limit of a countable inverse system. This example gives rise to a more general version of Question \ref{fed_countable}, namely:

\begin{question}
	Let $S = \{X_\beta, \pi^{\alpha}_{\beta}, \gamma\}$ be a continuous inverse system, where $\gamma$ is a countable ordinal, $X_0$ is a point, the neighboring bonding mappings $\pi^{\beta + 1}_{\beta}$ are fully closed, and for any $\beta + 1 < \gamma$ and any $y \in X_{\beta}$ the space $C\left((\pi^{\beta + 1}_{\beta})^{-1}(y)\right)$ is LUR renormable. Let $X$ be the limit $\lim\limits_{\leftarrow}S$. Does $C(X)$ admit an equivalent LUR norm?
\end{question}

\section*{Acknowledgments}
The author would like to thank Stanimir Troyanski for posing the problem and both him and Nadezhda Ribarska for the valuable discussions.
\printbibliography

\end{document}